\documentclass[12pt,reqno]{amsart}
\usepackage{amsthm}
\usepackage{fullpage,url}
\usepackage[off]{auto-pst-pdf}
\usepackage{hyperref} 
\usepackage{amscd}   
\usepackage[all,graph,poly, knot]{xy} 
\usepackage{amssymb,amsmath}
\usepackage{pstricks,pst-plot,pst-node}
\usepackage{multicol}
\usepackage{setspace}
\DeclareFontEncoding{OT2}{}{} 

\swapnumbers
\newtheorem{theorem}{Theorem}[section]

\theoremstyle{definition}
\newtheorem{definition}[theorem]{Definition}

\theoremstyle{remark}
\newtheorem{remark}[theorem]{Remark}
\newtheoremstyle{head}
{}
{}
{\bfseries}
{}
{}
{}
{.5em}
{}

\theoremstyle{head}

\newtheoremstyle{citing}
  {3pt}
  {3pt}
  {\itshape}
  {}
  {\bfseries}
  {:}
  {.5em}
  {\thmnote{#3}}

\theoremstyle{citing}
\begin{document}
\title[]{An inverted pendulum with a moving pivot point:\\ examples of topological approach}
\author{Ivan Polekhin}
\address{}
\email{ivanpolekhin@gmail.com}
\urladdr{}
\date{July 17, 2014}
\begin{abstract}
Two examples concerning an application of topology in the study of the dynamics of an inverted plain mathematical pendulum with a pivot point moving along a horizontal straight line are considered. 
The first example is an application of the Wa{\.z}ewski principle to the problem of the existence of a solution without falling in the case of a arbitrary prescribed law of motion of the pivot point. The second example is a proof of the existence of periodic solution in the same system when the law of motion is periodic as well. Moreover, in the second case it is also shown that along the obtained periodic solution the pendulum never becomes horizontal (falls). The proof is an example of application of the recent developments in the fixed point theory based on the Lefschetz-Hopf theory. 
\end{abstract}
\maketitle
%
%
%
%
\section{Introduction}
\label{section-introduction}
%
%
Both considered in the paper examples of the topological approach to the system of an inverted pendulum with a moving pivot point stem from a well-known problem originally presented in a famous \textit{What is mathematics?} book by Courant and Robbins \cite{CR} and known to be formulated by H. Whitney. The problem is as follows.
\par
\textit{Suppose a train travels from station $A$ to station $B$ along a straight 
section of track. The journey need not be of uniform speed or acceleration.
The train may act in any manner, speeding up, slowing down, 
coming to a halt, or even backing up for a while, before reaching $B$. 
But the exact motion of the train is supposed to be known in advance; 
that is, the function $s = f(t)$ is given, where $s$ is the distance of the train 
from station $A$, and $t$ is the time, measured from the instant of departure. 
On the floor of one of the cars a rod is pivoted so that it may move 
without friction either forward or backward until it touches the floor. If it 
does touch the floor, we assume that it remains on the floor henceforth; 
this will be the case if the rod does not bounce. Is it possible to place 
the rod in such a position that, if it is released at the instant when the 
train starts and allowed to move solely under the influence of gravity 
and the motion of the train, it will not fall to the floor during the entire 
journey from $A$ to $B$?}
\par
The positive answer is given by the authors as well as its correct explanation, which, however, is based on the following assumption:
\textit{the motion of the rod depends continuously on its initial condition.} This assumption seems natural
for a wide range of mechanical systems, yet it should be rigorously justified in the particular case since we assume that once the rod touches the floor, it remains on it henceforth, therefore, the continuity becomes less obvious. This shortcoming of the original proof was mentioned and briefly commented by Arnold in his book \cite{ARN} which also includes short overview of the articles related to the matter. Yet detailed consideration of the issue of whether original prove is full and correct or not is beyond the objective of the paper.
\par
In the first section, application of the topological Wa{\.z}ewski principle to the above problem is considered. This method allows us to prove existence of solutions without falling (even on an infinite time interval), including solutions with zero initial velocity of inverted pendulum. In the second section, we consider an application of recent developments in fixed point theory, based both on the Lefschetz-Hopf theorem and Wa\.{z}ewski's method, to the same mechanical system, and not only prove existence of solution without falling, but show that there always exists periodic solution without falling when it is assumed that the law of pivot point motion is periodic.  
%
%
\section{First example: Solution without falling}
\label{first-section}
%
In this section, we consider the system consisting of an inverted pendulum with moving pivot point and qualitatively study its dynamics. In particular, we prove existence of solutions without falling in case of arbitrary smooth law of motion of the pivot point. Though the proof is self-contained, if it is needed, one can find more detailed presentation of the Wa\.{z}ewski principle and related topics in \cite{WA}, \cite{HA}, and \cite{RSC}.
\par
Let $l$ be the distance between the pivot point and the mass point located at the end of the inverted pendulum (i.e. its length), the rod of the pendulum is massless and the mass point weighs $m$, the gravitational constant is denoted by $g$, and the law of motion of the pivot point along a straight horizontal line is given by a smooth function of time $f \colon [0, \infty) \to \mathbb{R}$. Therefore, the pendulum is moving in accordance with the gravity action and its dynamics also depends on the law of motion of the pivot point.
\par
Let $Oxy$ be a fixed Cartesian coordinate system, such that the pivot point moves along the $x$-axis and the $y$-axis is vertical and oriented in an opposite way to the gravitational force. Let $\varphi$ denote the angle between the $x$-axis and the rod ($\varphi = -\pi/2$ and $\varphi = \pi/2$ are horizontal positions for the pendulum), i.e. for $x$ and $y$ coordinates of the mass point we have
\begin{equation*}
\begin{aligned}
& x = f + l \sin\varphi,\\
& y = l \cos \varphi.
\end{aligned}
\end{equation*}
One can easily obtain the kinetic energy $T$ of the system
\begin{equation*}
T = \frac{m}{2}\left( \dot x^2 + \dot y^2 \right) = \frac{m}{2}\left( \dot f^2 + 2 \dot f l \dot \varphi \cos\varphi + l^2 \dot\varphi^2 \right),
\end{equation*}
and its potential energy
\begin{equation*}
U = mgy = mgl\cos\varphi.
\end{equation*}
Therefore, the Lagrangian function is as follows
\begin{equation*}
L = T - U = \frac{m}{2}\left( \dot f^2 + 2 \dot f l \dot \varphi \cos\varphi + l^2 \dot\varphi^2 \right) - mgl\cos\varphi.
\end{equation*}
Finally, we obtain the following system 
\begin{equation}
\label{main-eq}
\begin{aligned}
& \dot\varphi = p,\\
& \dot p = \frac{g}{l}\sin\varphi - \frac{\ddot f}{l}\cos\varphi.
\end{aligned}
\end{equation}
Here we assume that $\varphi$ variable is $2\pi$-periodic, i.e. we allow the pendulum to be under the horizontal floor, if the terminology of the original statement from \cite{CR} to be used.
\par
Let us now prove the following
\begin{theorem}
For the system (\ref{main-eq}) there exists $\varphi_0 \in (-\pi/2,\pi/2)$ such that the solution starting from $\varphi_0$ with $p_0 = 0$ at time $t = 0$ satisfies the following condition
\begin{equation*}
\label{main-ineq}
-\pi/2<\varphi(t,\varphi_0, 0)<\pi/2\quad\mbox{for all}\quad t \in [0,\infty).
\end{equation*}
\end{theorem}
\begin{proof}
Let $\Omega$ be the following subset of the extended phase space
\begin{equation*}
\Omega = \{ (t, \varphi, p) \in [0, \infty) \times \mathbb{R}/2\pi\mathbb{Z} \times \mathbb{R} \colon -\pi/2 \leqslant \varphi \leqslant \pi/2 \}.
\end{equation*}
Consider a line segment $L$ contained in $\Omega$ and defined in coordinates as follows
\begin{equation*}
L = \{ (t, \varphi, p) \in [0, \infty) \times \mathbb{R}/2\pi\mathbb{Z} \times \mathbb{R} \colon t = 0, -\pi/2 \leqslant \varphi \leqslant \pi/2, p = 0 \}.
\end{equation*}
We now show that $L$ contains at least one point such that the solution starting from it remains in the subset $-\pi/2<\varphi<\pi/2$ for all $t\geqslant 0$. Assume contrary. Then the following map, that we denote $\sigma$, from $L$ to $\partial\Omega$ is correctly defined
\begin{equation*}
\sigma \colon (0, \varphi_0, p_0) \in L \mapsto (t^*,\varphi(t^*,\varphi_0,p_0),p(t^*,\varphi_0,p_0)) \in \partial\Omega.
\end{equation*}
Where $t^* = \sup(T)$, $T = \{ s \in [0,\infty) \colon \varphi(t, \varphi_0, p_0) \in [-\pi/2,\pi/2] \,\mbox{for all}\,t\in[0,s] \}$.
\par
Let us now prove that $\sigma$ is continuous. In accordance with the Wa\.{z}ewski method, it can be shown by consideration of the system (\ref{main-eq}) in the vicinity of $\partial\Omega$. It is sufficient to show that any solution starting from $L$ is either transverse to the boundary $\partial\Omega$ at time $t>0$, or locally does not belong to $\Omega\setminus\partial\Omega$. Then the continuity follows from the continuous dependence from the initial conditions for \ref{main-eq}. For more details, see \cite{RSC}.
\par
\begin{figure}[h!]
\centering
\def\svgwidth{240 pt}
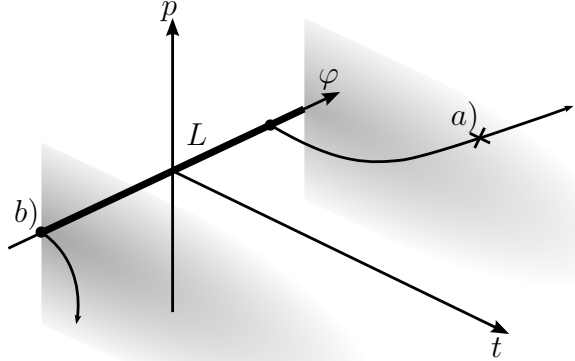
\caption{Each solution starting from $L$ and approaching half-planes $\varphi = \pm\pi/2$ is either a) transverse to the boundary $\partial\Omega$ at time $t>0$, or b) locally does not belong to $\Omega\setminus\partial\Omega$.}
\end{figure}
The latter sufficient condition for the continuity is satisfied for \ref{main-eq}. Indeed, for solutions intersecting $\Omega^+$ and $\Omega^-$, where
\begin{equation*}
\begin{aligned}
\Omega^- = &\{ (t, \varphi, p) \in [0, \infty) \times \mathbb{R}/2\pi\mathbb{Z} \times \mathbb{R} \colon t > 0, \varphi = \pi/2, p > 0 \}\cup\\
&\{(t, \varphi, p) \in [0, \infty) \times \mathbb{R}/2\pi\mathbb{Z} \times \mathbb{R} \colon t > 0, \varphi = -\pi/2, p < 0\},\\
\Omega^+ = &\{ (t, \varphi, p) \in [0, \infty) \times \mathbb{R}/2\pi\mathbb{Z} \times \mathbb{R} \colon t > 0, \varphi = \pi/2, p < 0 \}\cup\\
&\{(t, \varphi, p) \in [0, \infty) \times \mathbb{R}/2\pi\mathbb{Z} \times \mathbb{R} \colon t > 0, \varphi = -\pi/2, p > 0\},
\end{aligned}
\end{equation*}
we have transversality property satisfied. Moreover, solutions starting from $L$ cannot leave $\Omega$ through $\Omega^+$. Now prove that  at time $t>0$ they can only leave $\Omega$ through $\Omega^-$.
\par
For any solution starting from or intersecting $\partial\Omega$ at a point from the following set
\begin{equation*}
\begin{aligned}
\Omega^+_0 = &\{ (t, \varphi, p) \in [0, \infty) \times \mathbb{R}/2\pi\mathbb{Z} \times \mathbb{R} \colon t \geqslant 0, \varphi = \pi/2, p = 0 \}\cup\\
&\{(t, \varphi, p) \in [0, \infty) \times \mathbb{R}/2\pi\mathbb{Z} \times \mathbb{R} \colon t \geqslant 0, \varphi = -\pi/2, p = 0\},
\end{aligned}
\end{equation*}
From (\ref{main-eq}), we have
\begin{equation*}
\ddot \varphi = 
\begin{cases}
g/l &\mbox{if } \varphi=\pi/2\\
-g/l &\mbox{if } \varphi=-\pi/2.
\end{cases}
\end{equation*}
Therefore, we obtain that the solutions starting from $L$ cannot reach $\Omega^+_0$ at time $t>0$ and for $t=0$ they at least locally leave $\Omega$.
\par
Finally, we complete the proof by the following observation, typical to the Wa\.{z}ewski method. Consider the set $\omega$ of the boundary points which satisfies $\varphi=\pm\pi/2$, i.e. 
\begin{equation*}
\omega = \partial\Omega\setminus\{ (t, \varphi, p) \in [0, \infty) \times \mathbb{R}/2\pi\mathbb{Z} \times \mathbb{R} \colon \varphi\in(-\pi/2,\pi/2) \},
\end{equation*}
and \textit{the} map $\pi \colon \omega \to \omega \cap L$ that is continuous. Since $\sigma$ and $\pi$ are continuous, then we have a continuous map $\pi \circ \sigma$ from $L$ to its two-point boundary. This contradiction proves the theorem.
\end{proof}
\begin{remark}
In the proof, it was implicitly assumed that all solutions exist on $[0,\infty)$, yet for the considered system it can be rigorously proved.
\end{remark}
\begin{remark}
In the same way, one can prove the existence of solution without falling in the system of inverted spherical pendulum with pivot point moving on a horizontal plane. 
\end{remark}
%
%
\section{Second example: Periodic solution}
\label{second-section}
%
%
In this section we show that in the case of a periodic law of motion of the pivot point, there always exists a periodic solution without falling. This result is a straightforward application of some recent developments in the fixed point theory by Srzednicki, W\'{o}jcik and Zgliczynski \cite{SWZ} which we present here omitting details. First, following \cite{SWZ} we introduce some definitions which we slightly modify for our use.
\par
From now on, we assume that $v \colon \mathbb{R}\times M \to TM$ is a smooth time-dependent vector-field on a manifold $M$. 
\begin{definition}
For $t_0 \in \mathbb{R}$ and $x_0 \in M$, the map $t \mapsto x(t,t_0,x_0)$ is the solution for the initial value problem for the system $\dot x = v(t, x)$, such that $x(0,t_0,x_0)=x_0$.
\end{definition}
\begin{definition}
Let $W \subset \mathbb{R} \times M$. Define the {exit set} $W^-$ as follows. A point $(t,x)$ is in $W^-$ if there exists $\delta>0$ such that $(t+t_0, x(t,t_0,x_0)) \notin W$ for all $t \in (0,\delta)$.
\end{definition}
\begin{definition}
We call $W \subset \mathbb{R}\times M$ a {Wa\.{z}ewski block} for the system $\dot x = v(t,x)$ if $W$ and $W^-$ are compact.
\end{definition}
 Now introduce some notations. By $\pi_1$ and $\pi_2$ we denote the projections of $\mathbb{R}\times M$ onto $\mathbb{R}$ and $M$ respectively. If $Z \subset \mathbb{R}\times M$, $t\in\mathbb{R}$, then we denote
\begin{equation*}
Z_t=\{z \in M \colon (t,z) \in Z\}.
\end{equation*}
\begin{definition}
A set $W \subset [a,b] \times M$ is called a segment over $[a,b]$ if it is a block with respect to the system $\dot x = v(t,x)$ and the following conditions hold:
\begin{itemize}
\item there exists a compact subset $W^{--}$ of $W^-$ called the essential exit set such that
\begin{equation*}
W^-=W^{--}\cup(\{b\}\times W_b),\quad W^-\cap([a,b)\times M) \subset W^{--},
\end{equation*}
\item there exists a homeomorphism $h\colon [a,b]\times W_a \to W$ such that $\pi_1 \circ h = \pi_1$ and
\begin{equation}
\label{cond-2}
h([a,b]\times W_a^{--})=W^{--}.
\end{equation}
\end{itemize}
\end{definition}
\begin{definition}
Let $W$ be a segment over $[a,b]$. It is called periodic if
\begin{equation*}
(W_a,W_a^{--})=(W_b,W_b^{--}).
\end{equation*}
\end{definition}
\begin{definition}
For periodic segment $W$, we define the corresponding monodromy map $m$ as follows
\begin{equation*}
m\colon W_a\to W_a, \quad m(x) = \pi_2 h(b,\pi_2 h^{-1}(a,x)).
\end{equation*}
\end{definition}
\begin{remark}
The monodromy map $m$ is a homeomorphism. Moreover, it can be proved that a different choice of $h$ satisfying (\ref{cond-2}) leads to the monodromy map homotopic to $m$. It follows that the isomorphism in homologies
\begin{equation*}
\mu_W = H(m) \colon H(W_a,W_a^{--}) \to H(W_a, W_a^{--})
\end{equation*}
is an invariant of $W$.
\end{remark}
\begin{theorem} 
\label{main-th}
\cite{SWZ} Let W be a periodic segment over $[a,b]$. Then the set
\begin{equation*}
U = \{ x_0 \in W_a \colon x(t-a,a,x_0) \in W_t\setminus W_t^{--}\,\mbox{for all}\,\, t \in [a,b] \}
\end{equation*}
is open in $W_a$ and the set of fixed points of the restriction $x(b-a,a,\cdot)|_U \colon U \to W_a$ is compact. Moreover, if $W$ and $W^{--}$ are ANRs then
\begin{equation*}
\mathrm{ind}(x(b-a,a,\cdot)|_U) = \Lambda(m) - \Lambda(m|_{W_a^{--}}).
\end{equation*}
Where by $\Lambda(m)$ and $\Lambda(m|_{W_a^{--}})$ we denote the Lefschetz number of $m$ and $m|_{W_a^{--}}$ respectively. In particular, if $\Lambda(m) - \Lambda(m|_{W_a^{--}}) \ne 0$ then $x(b-a,a,\cdot)|_U$ has a fixed point in $W_a$.
\end{theorem}
Let us now continue with the following direct application of the above theorem to the system (\ref{main-eq}).
\begin{theorem}
Suppose that in (\ref{main-eq}) function $f \colon \mathbb{R} \to \mathbb{R}$ is $T$-periodic, then there exists $\varphi_0$ and $p_0$ such that for all $t \in \mathbb{R}$
\begin{enumerate}
\item $\varphi(t,0,\varphi_0,p_0) = \varphi(t+T,0,\varphi_0,p_0)$ and $p(t,0,\varphi_0,p_0) = p(t+T,0,\varphi_0,p_0)$ ,
\item $\varphi(t,0,\varphi_0,p_0) \in (-\pi/2,\pi/2)$.
\end{enumerate}
\end{theorem}
\begin{proof}
First, in order to apply \ref{main-th}, we show that a periodic Wa\.{z}ewski segment for our system can be defined as follows
\begin{equation*}
W = \{(t,\varphi,p) \in [0,T] \times \mathbb{R}/2\pi\mathbb{Z}\times\mathbb{R} \colon -\pi/2\leqslant\varphi\leqslant\pi/2, -p'\leqslant p \leqslant p'\},
\end{equation*}
where $p'$ satisfies
\begin{equation*}
\label{p-cond}
p' > \sup\limits_{t \in [0,T]} \frac{|\dddot\xi|}{g}.
\end{equation*}
It is clear that $W$ is compact. Let us show that $W^{--}$ is compact as well and
\begin{equation*}
\begin{aligned}
W^{--} = &\{ (t,\varphi,p) \in [0,T] \times \mathbb{R}/2\pi\mathbb{Z}\times\mathbb{R} \colon \varphi = \pi/2, 0 \leqslant p \leqslant p' \}\cup\\
&\{ (t,\varphi,p) \in [0,T] \times \mathbb{R}/2\pi\mathbb{Z}\times\mathbb{R} \colon \varphi = -\pi/2, -p' \leqslant p \leqslant 0 \}\cup\\
&\{ (t,\varphi,p) \in [0,T] \times \mathbb{R}/2\pi\mathbb{Z}\times\mathbb{R} \colon \varphi'(t) \leqslant \varphi \leqslant \pi/2, p = p' \}\cup\\
&\{ (t,\varphi,p) \in [0,T] \times \mathbb{R}/2\pi\mathbb{Z}\times\mathbb{R} \colon -\pi/2 \leqslant \varphi \leqslant \varphi'(t), p = -p' \},
\end{aligned}
\end{equation*}
where for a given $t \in [0,T]$, $\varphi'(t) \in (-\pi/2,\pi/2)$ satisfies
\begin{equation*}
g\sin\varphi'(t)-\ddot f (t)\cos\varphi'(t)=0.
\end{equation*}
Indeed, if $p=p'$ and $\varphi = \varphi'(t)$ (i.e. $\dot p = 0$) then we have from (\ref{main-eq})
\begin{equation*}
\begin{aligned}
\ddot p = &\frac{g}{l}p'\cos\varphi'(t) - \frac{\dddot f}{l}\cos\varphi'(t) + \frac{\ddot f}{l}p'\sin\varphi'(t)=\\
&\frac{g}{l}p'\cos\varphi'(t) - \frac{\dddot f}{l}\cos\varphi'(t) + \frac{\ddot f^2}{gl}p'\cos\varphi'(t)>0.
\end{aligned}
\end{equation*}
Since $\varphi'(t)$ is the only root for $\dot p = 0$ in $[-\pi/2,\pi/2]$, then for $p = p'$ and $\varphi'(t) < \varphi \leqslant \pi/2$ we obtain $\dot p > 0$, for $p = p'$ and $-\pi/2 \leqslant \varphi < \varphi'(t)$ we have $\dot p < 0$.
Similarly to the previous case, if $p=-p'$ and $\varphi = \varphi'(t)$ then
\begin{equation*}
\begin{aligned}
\ddot p = &-\frac{g}{l}p'\cos\varphi'(t) - \frac{\dddot f}{l}\cos\varphi'(t) - \frac{\ddot f}{l}p'\sin\varphi'(t)=\\
&-\frac{g}{l}p'\cos\varphi'(t) - \frac{\dddot f}{l}\cos\varphi'(t) - \frac{\ddot f^2}{gl}p'\cos\varphi'(t)<0,
\end{aligned}
\end{equation*}
and $\dot p > 0$ when $p=-p'$, $\varphi'(t)<\varphi\leqslant\pi/2$; $\dot p < 0$ when $p=p'$, $-\pi/2\leqslant\varphi<\varphi'(t)$.
\par
For the rest part of $W^{--}$ we have already shown in the previous section that $\ddot \varphi > 0$ if $\varphi = \pi/2$, $p = \dot \varphi = 0$, and $\ddot \varphi < 0$ if $\varphi = -\pi/2$, $p = \dot \varphi = 0$.
\par
\begin{figure}[h!]
\centering
\def\svgwidth{320 pt}
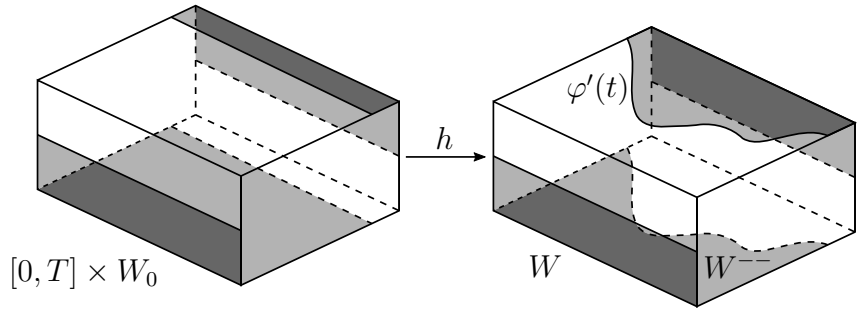
\caption{Periodic segment $W$. $W^{--}$ is in gray.}
\end{figure}
Omitting explicit definition of a homeomorphism $h$, we note that $m = \mathrm{id}$. Therefore,
\begin{equation*}
\Lambda(\mathrm{id}|_{W_0}) - \Lambda(\mathrm{id}|_{W_0^{--}}) = \chi(W_0) - \chi(W_0^{--}) = -1,
\end{equation*}
and theorem \ref{main-th} can be applied.
\end{proof}
%
%
%
%


\begin{thebibliography}{}
%
%
%
\bibitem[CR]{CR}
R. Courant, H. Robbins,
{\it What is mathematics?: an elementary approach to ideas and methods},
Oxford University Press (1996).
\bibitem[Arn]{ARN}
V. Arnold
{\it What is mathematics?}
MCCME, Moscow, (2002) (In Russian) 
\bibitem[Wa]{WA}
T. Wa\.{z}ewski,
{\it Sur un principe topologique de l'examen de l'allure asymptotique des int\'{e}grales des \'{e}quations diff\'{e}rentielles ordinaires},
Ann. Soc. Polon. Math. 20 (1947), 279--313.
\bibitem[Ha]{HA}
P. Hartman,
{\it Ordinary differential equations},
Classics in Applied Mathematics 38 (1964).
\bibitem[RSC]{RSC}
R. Reissig, G. Sansone, R. Conti, 
{\it Qualitative theorie nichtlinearer differentialgleichungen},
Edizione Cremonese, (1963).
\bibitem[SWZ]{SWZ}
R. Srzednicki, K. W\'{o}jcik, and P. Zgliczynski,
{\it Fixed point results based on Wa\.{z}ewski method} in {\it Handbook of topological fixed point theory}, Ed: R. Brown, M. Furi, L. G\'{o}rniewicz, B. Jiang (2005), 903--941.
\end{thebibliography}
\end{document}